\documentclass[11pt,letterpaper,reqno]{amsart}
\usepackage[includehead,includefoot,margin=30mm]{geometry}
\usepackage{comment}
\usepackage{times}
\usepackage{amsfonts}
\usepackage{amsmath}
\usepackage{amssymb}
\usepackage{enumerate}
\usepackage{amsthm}
\usepackage[all]{xy}
\usepackage{graphicx}
\usepackage{nicefrac}
\usepackage{color}
\usepackage{hyperref}

\hypersetup{
    colorlinks=true,
    linkcolor=blue,
    citecolor=blue,
    filecolor=blue,
    urlcolor=blue
}
\usepackage[nameinlink,capitalize,noabbrev]{cleveref}

\usepackage{calrsfs}

\newtheorem{theorem}{Theorem}[section]

\newtheorem{lemma}[theorem]{Lemma}

{\theoremstyle{remark}
  \newtheorem{remark}[theorem]{Remark}}
{\theoremstyle{definition}

}

\newtheorem{introthm}{Theorem}

\newcommand{\ZZ}[0]{\ensuremath{\mathbf{Z}}}
\newcommand{\kk}[0]{\ensuremath{\mathbf{k}}}

\newcommand{\Char}[0]{\ensuremath{\operatorname{char}}}

\newcommand{\RR}[0]{\ensuremath{\mathbf{R}}}

\begin{document}

\title[A one-step counterexample to the normalized Nash blowup conjecture]{A one-step counterexample to the \\ normalized Nash blowup conjecture}

\author[A. Liendo]{Alvaro Liendo} %
\address{Instituto de Matem\'aticas, Universidad de Talca, Talca, Chile}  %
\email{aliendo@utalca.cl}

\author[A.J. Palomino]{Ana Julisa Palomino}
\address{Instituto de Matem\'aticas, Universidad de Talca, Talca, Chile}  %
\email{ana.palomino@utalca.cl}

\author[G. Rodr\'iguez]{Gonzalo Rodr\'iguez} %
\address{Instituto de Matem\'aticas, Universidad de Talca, Talca, Chile}  %
\email{gonzalo.rodriguez@utalca.cl}

\date{\today}

\thanks{{\it 2020 Mathematics Subject
    Classification}: 14B05; 14E15; 14M25.\\
  \mbox{\hspace{11pt}}{\it Key words}: Normalized Nash blowup, resolution of singularities, toric varieties, positive characteristic.\\
  \mbox{\hspace{11pt}} The first author was partially supported by Fondecyt Project 1240101 and 
by Fondecyt Exploraci\'on Project 13250049. The second author was partially 
supported by CONICYT-PFCHA/Doctorado Nacional, folio 21240560, and by 
Fondecyt Exploraci\'on Project 13250049. The third author was partially 
supported by CONICYT-PFCHA/Doctorado Nacional, folio 21210992.}

\begin{abstract}
We construct an explicit normal singular affine toric variety $X$ of
dimension five over an algebraically closed field of characteristic three
such that the normalized Nash blowup of $X$ already contains an open
affine subset isomorphic to~$X$. Combined with previously known examples,
this yields one-step counterexamples in every dimension greater than or
equal to five and every characteristic. The characteristic-three case is
the most delicate: the previously known counterexample in dimension four
requires a two-step iteration of the normalized Nash blowup, and our
example demonstrates that in dimension five and higher the minimal number
of iterations needed to produce a loop is one.
\end{abstract}

\maketitle

\section*{Introduction}

Let $X \subseteq \kk^n$ be an equidimensional algebraic variety of
dimension $d$ over an algebraically closed field $\kk$.
The \emph{Nash blowup} of $X$, introduced independently by
J.\,G.~Semple \cite{Se54} and J.~Nash \cite{Sp90}, is the proper
birational map $\nu\colon X^* \to X$ obtained by taking the Zariski
closure of the graph of the Gauss map
\[
  \Phi\colon X\setminus \operatorname{Sing}(X)\longrightarrow
    \operatorname{Grass}(d,n),\quad x\mapsto T_x X,
\]
together with the natural projection.  The map $\nu$ is an isomorphism
over the smooth locus.  One can also compose $\nu$ with the
normalization to obtain the \emph{normalized Nash blowup}
$\overline{\nu}\colon \overline{X^*}\to X$.

Both Nash and Semple proposed to resolve singularities by iterating
these constructions.  The question has attracted sustained attention
over more than half a century. Nobile proved that the Nash blowup
is an isomorphism if and only if $X$ is smooth, in characteristic
zero~\cite{Nob75}; the analogous statement for the normalized Nash
blowup in arbitrary characteristic was established by Duarte and
N\'u\~nez Betancourt~\cite{DuNu22}. Positive results for specific
families of varieties were obtained in \cite{Nob75, Reb, GS1, GS2,
Hi83, Sp90, GS3, Ataetal, GrMi12, GPTe14, DuarSurf, DG, DJN24, DDR,
CDLLcharfree}.
\medskip

Negative answers to the conjecture were given in \cite{CDLAL} and
\cite{CDLL3d}.  In \cite{CDLAL}, it is shown that in every dimension
$d\geq 4$ and every characteristic, there exist normal singular affine
toric varieties for which neither iterating the Nash blowup nor
iterating the normalized Nash blowup resolves the singularities.
Subsequently, \cite{CDLL3d} showed that the (non-normalized) Nash blowup
conjecture also fails in dimension three: there exists a non-normal
affine toric variety $X$ of dimension $3$ over a field of characteristic
zero such that the second iteration of the Nash blowup of $X$ contains
an open affine subset isomorphic to~$X$.
More precisely, the following is proved in \cite{CDLAL}.

\begin{introthm}[{\cite{CDLAL}}]\label{thm:CDLAL}
For every $d\geq 4$ and every algebraically closed field $\kk$,
there exists a normal singular affine algebraic variety $X$ of
dimension $d$ over $\kk$ such that:
\begin{itemize}
  \item[$(i)$] If $\Char(\kk)=0$, then the Nash blowup and the
    normalized Nash blowup of $X$ each contain an open affine subset
    isomorphic to $X$.
  \item[$(ii)$] If $\Char(\kk)$ is positive and different from $3$,
    then the normalized Nash blowup of $X$ contains an open affine
    subset isomorphic to $X$.
  \item[$(iii)$] If $\Char(\kk)=3$, then the \emph{second} iteration
    of the normalized Nash blowup of $X$ contains an open affine
    subset isomorphic to $X$.
\end{itemize}
\end{introthm}

A notable feature of Theorem~\ref{thm:CDLAL} is the asymmetry in
characteristic three: while every other characteristic admits a
\emph{one-step} counterexample (a variety that reappears already after
the first normalized Nash blowup), characteristic three required a
\emph{two-step} loop, at least in dimension four.  The counterexample in that case is the normal affine toric variety
associated with the cone generated in $\RR^4$ by the columns of the matrix
\[
  \left[
  \begin{array}{rrrrr}
    1 & 0 & 1 & 0 & 0 \\
    0 & 1 & 2 & 0 & 3 \\
    0 & 0 & 3 & 0 & 3 \\
    0 & 0 & 0 & 1 & 1
  \end{array}\right],
\]
which arises as a chart in the fourth iteration of the normalized Nash blowup
of the toric variety associated with the Reeves cone, that is,
the cone generated by the columns of the matrix
\[
  \left[
  \begin{array}{rrrr}
    1 & 0 & 0 & 1 \\
    0 & 1 & 0 & 1 \\
    0 & 0 & 1 & 1 \\
    0 & 0 & 0 & 5
  \end{array}\right].
\]
Extensive computer searches conducted by the authors of the present paper
did not reveal any four-dimensional one-step counterexample in
characteristic three. The question therefore arises naturally:
\emph{Does every characteristic admit a one-step counterexample
in some dimension?} 

The following theorem is the main result of the present paper, which answers
this question affirmatively for every characteristic in dimension five and higher.

\begin{introthm}\label{thm:main}
For every $d \geq 5$ and every algebraically closed field $\kk$,
there exists a normal singular affine toric variety $X$ of dimension
$d$ over $\kk$ such that the normalized Nash blowup of $X$ contains
an open affine subset isomorphic to $X$.
\end{introthm}

If $\Char(\kk) \neq 3$, this follows directly from
Theorem~\ref{thm:CDLAL} (i) and~(ii). If $\Char(\kk) = 3$,
we construct an explicit five-dimensional variety $X$ satisfying
the theorem, which combined with \cite[Lemma~1]{CDLAL} achieves
the proof. This variety $X$ is the normal affine toric variety
associated with the cone $\omega \subset \mathbf{R}^5$ generated
by the columns of the matrix
\[
  B = \left[
  \begin{array}{rrrrrrrr}
    1 & 0 & 0 & 0 & 0 &  2 &  1 &  1 \\
    0 & 1 & 0 & 0 & 0 &  2 &  2 &  2 \\
    0 & 0 & 1 & 0 & 0 & -1 & -1 &  0 \\
    0 & 0 & 0 & 1 & 0 &  1 &  1 &  0 \\
    0 & 0 & 0 & 0 & 1 & -2 & -1 & -1
  \end{array}\right].
\]
The proof consists of an
explicit combinatorial computation in the spirit of \cite{CDLAL},
using the description of the normalized Nash blowup of a toric variety
due to Gonz\'alez-Sprinberg \cite{GS1} and Duarte, Jeffries, and
N\'u\~nez Betancourt \cite{DJN24}.

Compared with the characteristic-three example in dimension four from
\cite{CDLAL}, the present example has two advantages: it produces a
loop already at the first iteration, and it is normal (so it also
provides a counterexample to the normalized Nash blowup conjecture for
normal varieties).

On the non-normal side, a three-dimensional counterexample to the
non-normalized Nash blowup conjecture in characteristic zero was
recently obtained in \cite{CDLL3d}, where a non-normal toric variety
of dimension three is shown to re-appear after two iterations of the
(non-normalized) Nash blowup.

The counterexample in the present paper was found via computer experimentation
following the computational strategy developed in \cite{CDLAL,CDLLcomputational},
using the software SageMath~\cite{sagemath}. Concretely, we populated a directed
graph whose vertices are affine semigroups and whose edges record the normalized
Nash blowup operation, and searched for cycles of length one within it.

\section{Preliminaries on toric varieties and Nash blowups}

We gather here the background needed for the proof of
Theorem~\ref{thm:main}.  Standard references for toric varieties are
\cite{fulton1993introduction, oda1983convex, cox2011toric}; we follow the conventions of \cite{cox2011toric},
and in particular we do \emph{not} require toric varieties to be normal.

\subsection*{Affine semigroups and toric varieties}

An \emph{affine semigroup} is a finitely generated, cancellative,
commutative monoid that embeds into a free abelian group.  We fix a
lattice $M \cong \ZZ^d$ and without loss of generality we assume that $S \subset M$ and that the group generated by $S$ is $M$. Let $M_{\RR}:= M \otimes_{\ZZ}\RR$.

The \emph{polyhedral cone} generated by $S$ is
\[
  \omega = \operatorname{Cone}(S)
    = \left\{\,\sum_{u \in F} \lambda_u\, u \;\Big|\;
      F \subset S \text{ finite},\; \lambda_u \geq 0\,\right\}
    \subset M_{\RR}\ .
\]
We say that $S$ is \emph{pointed} if $\omega$ is strongly convex, that is,
$\omega \cap (-\omega) = \{0\}$.  We say that $S$ is
\emph{saturated}  if $\omega \cap M = S$.

A pointed semigroup admits a unique minimal generating set, called its
\emph{Hilbert basis} $\mathcal{H}(S)$; it consists of the elements of
$S$ that cannot be expressed as a sum of two nonzero elements of $S$.

Given an affine semigroup $S$ one associates the \emph{semigroup
algebra}
\[
  \kk[S] = \bigoplus_{u \in S} \kk\cdot\chi^u,
    \qquad \chi^0 = 1,\quad
    \chi^u \cdot \chi^{u'} = \chi^{u+u'},
\]
and the affine toric variety $X(S) := \operatorname{Spec}\,\kk[S]$.
The variety $X(S)$ is normal if and only if $S$ is saturated
\cite[Theorem~1.3.5]{cox2011toric}.

\subsection*{Nash blowup and normalized Nash blowup of a toric variety}

The combinatorial description of the Nash blowup and normalized Nash
blowup of an affine toric variety in arbitrary characteristic is
due to Gonz\'alez-Sprinberg \cite{GS1} (characteristic zero, normal
case), Gonz\'alez P\'erez--Teissier \cite{GPTe14} (characteristic zero,
general case), and Duarte--Jeffries--N\'u\~nez Betancourt \cite{DJN24}
(prime characteristic).  We recall the description following
\cite[Section~1.9.2]{Sp20} and \cite{CDLAL}.

Let $X(S)$ be an affine toric variety given by a pointed semigroup
$S \subset M$, and let $\mathcal{H}(S) = \{h_1,\ldots,h_r\}$ be its Hilbert basis.
For a subset of $d$ elements $\{h_{i_1},\ldots,h_{i_d}\} \subset \mathcal{H}(S)$,
define the matrix $(h_{i_1}\cdots h_{i_d})$ whose columns are the vectors $h_{i_j}$. Let $p \geq 0$ be the characteristic of $\kk$. We denote
\[
  {\det}_p(h_{i_1}\cdots h_{i_d})
  = \begin{cases}
      \det(h_{i_1}\cdots h_{i_d}) & \text{if }p = 0 \ , \\
      \det(h_{i_1}\cdots h_{i_d}) \bmod p & \text{if }p > 0 \ .
    \end{cases}
\]

The affine charts of the Nash blowup and the normalized Nash blowup are
indexed by the subsets $A = \{h_{i_1},\ldots,h_{i_d}\} \subset \mathcal{H}(S)$
satisfying $\det_p(h_{i_1}\cdots h_{i_d}) \neq 0$.
For such a subset $A$ and each $h \in A$, without loss of generality, up to
reordering the indices, we may and will assume $A=\{h_1,\ldots,h_d\}$ and
$h= h_1$. We define
\begin{equation}\label{eq:GA}
  \mathcal{G}_A(h)
  = \left\{g - h \;\big|\;
      g \in \mathcal{H}(S)\setminus A,\;
      {\det}_p(g\; h_{2}\cdots h_{d}) \neq 0 \right\}.
\end{equation}
Computing $\mathcal{G}_A(h)$ for all $d$ elements of $A$, we define
$\mathcal{G}_A = \mathcal{H}(S) \cup \mathcal{G}_A(h_{1}) \cup \cdots
\cup \mathcal{G}_A(h_{d})$.
Let $S_A$ be the semigroup generated by $\mathcal{G}_A$ in $M$, and let
$\overline{S_A} = \operatorname{Cone}(S_A) \cap M$ be its saturation.
The affine toric varieties $X(S_A)$ and $X(\overline{S_A})$, taken over
all $A$ with
\begin{align*}
   {\det}_p(h_{i_1}\cdots h_{i_d}) \neq 0 \quad \text{and}\quad S_A \,\text{ pointed},
\end{align*}
form sets of covering affine charts of the Nash blowup and the normalized
Nash blowup of $X(S)$, respectively
(see \cite[Propositions~32 and~60]{GPTe14} for the characteristic zero
case and \cite[Theorem~1.9]{DJN24} for the prime characteristic case).

\section{Proof of the theorem}\label{sec:proof}

To achieve the proof of Theorem~\ref{thm:main} in the case $\Char(\kk)=3$,
we exhibit a pointed, saturated affine semigroup $S \subset \ZZ^5$ such
that one of the charts of the normalized Nash blowup of the normal affine
toric variety $X(S)$ is isomorphic to $X(S)$ itself, over any algebraically
closed field of characteristic~$3$.

Let $M = \ZZ^5$ and let $\omega \subset M_{\RR}$ be the cone generated
by the columns of the matrix
\begin{align*}
     B = \left[
     \begin{array}{rrrrrrrr}
    1 & 0 & 0 & 0 & 0 &  2 &  1 &  1 \\
    0 & 1 & 0 & 0 & 0 &  2 &  2 &  2 \\
    0 & 0 & 1 & 0 & 0 & -1 & -1 &  0 \\
    0 & 0 & 0 & 1 & 0 &  1 &  1 &  0 \\
    0 & 0 & 0 & 0 & 1 & -2 & -1 & -1
    \end{array}\right]\ .
\end{align*}

Set $S = \omega \cap M$ and $X = X(S)$.
We denote by $h_i$ the vector corresponding to the $i$-th column of
$B$, for $i \in \{1,\ldots,8\}$.

\begin{lemma}\label{lem:pointed}
  The affine semigroup $S$ is pointed.
\end{lemma}

\begin{proof}
  The linear functional $\mathcal{L}(x_1,\ldots,x_5) = x_1+x_2+x_3+x_4+x_5$
  satisfies $\mathcal{L}(h_i) = 1$ for $i \in \{1,\ldots,5\}$ and
  $\mathcal{L}(h_j) \geq 2$ for $j \in \{6,7,8\}$.
  In particular $\mathcal{L}$ is strictly positive on $S\setminus\{0\}$, which forces $S$ to be pointed.
\end{proof}

\begin{lemma}\label{lem:hilbert}
  The Hilbert basis of $S$ is $\mathcal{H}(S) = \{h_1,\ldots,h_8,h_9\}$,
  where $h_9 = (1,1,0,1,-1)$.
\end{lemma}

\begin{proof}
Let $\omega= \operatorname{Cone}(h_1,\ldots,h_8)$, and let $H=\{h_1,\ldots,h_8,h_9\}\subset M$. We now prove that $H$ is the Hilbert basis of $S$.

Consider the simplicial subdivision of $\omega$ given by the cones:
    \begin{align*}
        \sigma_1&=\operatorname{Cone} (h_1,h_2,h_3,h_4,h_5)\ , & \quad &  \sigma_5=\operatorname{Cone}(h_1,h_2,h_3,h_4,h_9)\ , \\
        \sigma_2&=\operatorname{Cone}(h_1,h_2,h_4,h_5,h_7)\ , & \quad & \sigma_6=\operatorname{Cone}(h_1,h_2,h_4,h_6,h_9)\ , \\
        \sigma_3&= \operatorname{Cone}(h_1,h_2,h_3,h_6,h_8)\ , & \quad &  \sigma_7= \operatorname{Cone}(h_1,h_2,h_3,h_6,h_9)\ . \\
        \sigma_4&= \operatorname{Cone}(h_1,h_2,h_4,h_6,h_7)\ ,     
    \end{align*}
Observe that the above subdivision is in fact unimodular, that is, $\det(\sigma_i)=1$ for $i\in \{1,\ldots,7\}$. Therefore, the set $\{h_1,\ldots,h_8,h_9\}$ generates the semigroup $S$. It remains to prove that the set $H$ is minimal. Since each $h_i$ ($i=1,\ldots,8$) is a primitive ray generator of $\omega$, it belongs to $\mathcal{H}(S)$. The linear functional $\mathcal{L}$ from Lemma~\ref{lem:pointed} satisfies
$\mathcal{L}(h_9)=2$, so if $h_9 = u + v$ with $u,v \in S\setminus\{0\}$,
then $\mathcal{L}(u)=\mathcal{L}(v)=1$. The only elements of $S$ with
$\mathcal{L}$-value equal to $1$ are $h_1,\ldots,h_5$, but no sum of two
of these equals $h_9=(1,1,0,1,-1)$. Thus $h_9$ is part of the Hilbert basis.
\end{proof}

Observe that
\[
  {\det}_3(h_1,h_2,h_4,h_5,h_6) = 2 \pmod 3.
\]
Set $A = \{h_1,h_2,h_4,h_5,h_6\} \subset \mathcal{H}(S)$.
We will show that the chart $X(\overline{S_A})$ of the normalized Nash
blowup of $X$ is isomorphic to $X$ itself, which gives a one-step loop.

\medskip
We have $\mathcal{H}(S)\setminus A = \{h_3,h_7,h_8,h_9\}$.
The nonzero determinant conditions (where at the start of each block we
indicate the element of $A$ being replaced) are as follows:
\begin{align*}
  h_1\colon\quad
  &{\det}_3(h_3,h_2,h_4,h_5,h_6)\neq 0,\quad
   {\det}_3(h_7,h_2,h_4,h_5,h_6)\neq 0, \\
  &{\det}_3(h_8,h_2,h_4,h_5,h_6)\neq 0,\quad
   {\det}_3(h_9,h_2,h_4,h_5,h_6)\neq 0. \\[4pt]
  h_2\colon\quad
  &{\det}_3(h_1,h_3,h_4,h_5,h_6)\neq 0,\quad
   {\det}_3(h_1,h_7,h_4,h_5,h_6)=0, \\
  &{\det}_3(h_1,h_8,h_4,h_5,h_6)\neq 0,\quad
   {\det}_3(h_1,h_9,h_4,h_5,h_6)\neq 0. \\[4pt]
  h_4\colon\quad
  &{\det}_3(h_1,h_2,h_3,h_5,h_6)\neq 0,\quad
   {\det}_3(h_1,h_2,h_7,h_5,h_6)=0, \\
  &{\det}_3(h_1,h_2,h_8,h_5,h_6)=0,\quad
   {\det}_3(h_1,h_2,h_9,h_5,h_6)\neq 0. \\[4pt]
  h_5\colon\quad
  &{\det}_3(h_1,h_2,h_4,h_3,h_6)\neq 0,\quad
   {\det}_3(h_1,h_2,h_4,h_7,h_6)\neq 0, \\
  &{\det}_3(h_1,h_2,h_4,h_8,h_6)\neq 0,\quad
  {\det}_3(h_1,h_2,h_4,h_9,h_6)\neq 0. \\[4pt]
  h_6\colon\quad
  &{\det}_3(h_1,h_2,h_4,h_5,h_3)\neq 0,\quad
   {\det}_3(h_1,h_2,h_4,h_5,h_7)\neq 0, \\
  &{\det}_3(h_1,h_2,h_4,h_5,h_8)=0,\quad
   {\det}_3(h_1,h_2,h_4,h_5,h_9)=0.
\end{align*}

From \eqref{eq:GA} we obtain
\begin{align*}
  \mathcal{G}_A
  &= \mathcal{H}(S)
     \;\cup\; \{h_3-h_1,\, h_7-h_1,\, h_8-h_1,\, h_9-h_1\} \\
  &\phantom{{}= \mathcal{H}(S)}
     \;\cup\; \{h_3-h_2,\, h_8-h_2,\, h_9-h_2\}  \\
  &\phantom{{}= \mathcal{H}(S)}
     \;\cup\; \{h_3-h_4,\, h_9-h_4\} \\
  &\phantom{{}= \mathcal{H}(S)}
     \;\cup\; \{h_3-h_5,\, h_7-h_5,\, h_8-h_5,\, h_9-h_5\} \\
  &\phantom{{}= \mathcal{H}(S)}
     \;\cup\; \{h_3-h_6,\, h_7-h_6\}.
\end{align*}

Consider the set
\[
  H = \{h_1,\; h_3-h_2,\; h_9-h_2,\; h_3-h_4,\; h_9-h_4,\;
         h_3-h_5,\; h_7-h_5,\; h_3-h_6,\; h_7-h_6\}.
\]
Let $R$ denote the semigroup generated by $H$.
We claim that $R = S_A$.  Since $H \subset \mathcal{G}_A$ we have
$R \subset S_A$.  Conversely, note that
$h_8-h_2 = h_9-h_4$ (which lies in $H$), and that every other element
of $\mathcal{G}_A$ decomposes as a sum of elements of $H$:
\begin{align*}
   h_2    &= (h_7-h_6) + (h_9-h_4),  &    h_9-h_5 &= (h_7-h_5) + (h_3-h_2), \\
   h_4    &= (h_7-h_6) + (h_9-h_2),  &    h_3     &= (h_3-h_5) + h_5, \\
   h_5    &= (h_7-h_6) + h_1,        &    h_7     &= (h_7-h_5) + h_5, \\
   h_6    &= (h_7-h_5) + h_1,        &    h_8-h_1 &= (h_7-h_1) + (h_3-h_4) , \\
   h_3-h_1 &= (h_3-h_5) + (h_7-h_6), &    h_9-h_1 &= (h_7-h_1) + (h_3-h_2), \\
   h_7-h_1 &= (h_7-h_5) + (h_7-h_6), &    h_8     &= (h_8-h_1) + h_1, \\
   h_8-h_5 &= (h_7-h_5) + (h_3-h_4), &    h_9     &= (h_9-h_1) + h_1.
\end{align*}

Hence $S_A \subset R$, so indeed $R = S_A$.

It remains to show that $\overline{S_A} \cong S$.
Consider the automorphism $U\colon M \to M$ given by the unimodular matrix
\[
  U =\left[ \begin{array}{rrrrr}
   -1 &  0 &  1 &  0 & -2 \\
    0 & -1 &  0 &  0 & -2 \\
    0 &  1 &  0 &  1 &  2 \\
    0 &  0 &  0 & -1 & -1 \\
    1 &  0 &  0 &  0 &  2
  \end{array}\right] \ .
\]
A direct computation yields that $U$ maps the generators of $S$
bijectively onto the generators of $S_A$:
\begin{align*}
  U(h_1) &= h_7-h_6, & U(h_4) &= h_3-h_4, & U(h_7) &= h_3-h_5, \\
  U(h_2) &= h_3-h_2, & U(h_5) &= h_3-h_6, & U(h_8) &= h_9-h_2, \\
  U(h_3) &= h_1,     & U(h_6) &= h_7-h_5, & U(h_9) &= h_9-h_4.
\end{align*}
Since $U$ is an automorphism of the lattice $M$, it induces an
isomorphism between the semigroups $S_A$ and $S$.
Since $S$ is pointed and saturated, so is $S_A$, which yields $S_A = \overline{S_A}$.
In particular, $X(\overline{S_A})$ is an affine chart of the normalized Nash
blowup of $X(S)$ isomorphic to $X(S)$ itself,
which is precisely the statement of Theorem~\ref{thm:main}. \qed

\begin{remark}
Since $S_A$ is pointed and saturated, the chart $X(\overline{S_A})=X(S_A)$ is
isomorphic to $X(S)$ not only as a chart of the \emph{normalized} Nash
blowup, but already as a chart of the (non-normalized) Nash blowup of
$X(S)$. Hence $X(S)$ is also a counterexample to the Nash blowup
conjecture in characteristic three. This should be contrasted with the
classical example of Nobile \cite{Nob75}, namely the curve $x^p - y^q = 0$
in characteristic $p$ with $p\neq q$, whose Nash blowup is isomorphic to itself
but which is non-normal. The variety $X(S)$ constructed here provides a
\emph{normal} counterexample to the Nash blowup conjecture in positive
characteristic.
\end{remark}

\begin{remark}
    Let $\kk[x_1,\ldots,x_9]$ be the polynomial ring in nine variables. The map $\kk[x_1\ldots,x_9] \to \kk[S]$ given by $x_i \mapsto \chi^{h_i}$ is surjective. Its kernel $I$ is generated by
    \begin{align*}
        x_9^2-x_3x_4x_6\ ,& & x_7x_8-x_2^2x_6 \ , &     &x_1x_7-x_5x_6 \ , &  &x_7x_9-x_2x_4x_6 \ ,& &
         x_8x_9-x_2x_3x_6 \ , \\   x_5x_9-x_1x_2x_4 \ , & &  x_5x_8-x_1x_2^2 \ , &  & x_4x_8-x_2x_9 \ , & &  x_3x_7-x_2x_9 \ , &  & x_3x_5x_6-x_1x_2x_9 \ .
    \end{align*}
    Therefore, the toric variety $X(S)$, whose normalized Nash blowup contains an affine chart isomorphic to itself, is realized as the zero locus of $I$ in $\kk^9$. Moreover, the singular locus of $X(S)$, in the coordinates $x_1,\ldots,x_9$, is the union of the following linear subspaces:
    \begin{align*}
        & \{x_1=x_2=x_5=x_6=x_7=x_8=x_9=0\}\ , \quad  \{x_1=x_3=x_4=x_6=x_7=x_8=x_9=0\} \ ,\\
        & \{ x_2=x_4=x_5=x_6=x_7=x_8=x_9=0\} \ , \quad \{x_2=x_3=x_6=x_7=x_8=x_9=0\}\ , \quad \text{and}\\
       &
        \{x_2=x_3=x_4=x_5=x_7=x_8=x_9=0\} \ . 
    \end{align*}
\end{remark}
\bibliographystyle{alpha}
\bibliography{ref}

\newcommand{\etalchar}[1]{$^{#1}$}
\begin{thebibliography}{CDLAL25b}

\bibitem[ALP{\etalchar{+}}11]{Ataetal}
Atanas Atanasov, Christopher Lopez, Alexander Perry, Nicholas Proudfoot, and Michael Thaddeus.
\newblock Resolving toric varieties with {N}ash blowups.
\newblock {\em Exp. Math.}, 20(3):288--303, 2011.

\bibitem[CDLAL25a]{CDLLcharfree}
Federico Castillo, Daniel Duarte, Maximiliano Leyton-\'Alvarez, and Alvaro Liendo.
\newblock Characteristic-free normalized nash blowup of toric varieties.
\newblock {\em arXiv:2501.09811}, 2025.

\bibitem[CDLAL25b]{CDLL3d}
Federico Castillo, Daniel Duarte, Maximiliano Leyton-\'Alvarez, and Alvaro Liendo.
\newblock Non-normalized nash blowup fails to resolve singularities in dimension three.
\newblock {\em arXiv:2511.01772}, 2025.

\bibitem[CDLAL25c]{CDLLcomputational}
Federico Castillo, Daniel Duarte, Maximiliano Leyton-\'Alvarez, and Alvaro Liendo.
\newblock On a computational approach to the nash blowup problem.
\newblock {\em arXiv:2511.17862}, 2025.

\bibitem[CDLAL26]{CDLAL}
Federico Castillo, Daniel Duarte, Maximiliano Leyton-\'Alvarez, and Alvaro Liendo.
\newblock Nash blowup fails to resolve singularities in dimensions four and higher.
\newblock {\em Annals of Mathematics}, 203(2):677--694, March 2026.

\bibitem[CLS11]{cox2011toric}
David~A. Cox, John~B. Little, and Henry~K. Schenck.
\newblock {\em Toric varieties}, volume 124.
\newblock American Mathematical Soc., 2011.

\bibitem[DDR25]{DDR}
Tha\'is~M. Dalbelo, Daniel Duarte, and Maria Aparecida~Soares Ruas.
\newblock Nash blowups of 2-generic determinantal varieties in positive characteristic.
\newblock {\em Pure Appl. Math. Q.}, 21(4):1557--1575, 2025.

\bibitem[DJNnB24]{DJN24}
Daniel Duarte, Jack Jeffries, and Luis N\'u\~nez Betancourt.
\newblock Nash blowups of toric varieties in prime characteristic.
\newblock {\em Collect. Math.}, 75(3):629--637, 2024.

\bibitem[DNnB22]{DuNu22}
Daniel Duarte and Luis N\'u\~nez Betancourt.
\newblock Nash blowups in prime characteristic.
\newblock {\em Rev. Mat. Iberoam.}, 38(1):257--267, 2022.

\bibitem[DT18]{DG}
Daniel Duarte and Daniel~Green Tripp.
\newblock Nash modification on toric curves.
\newblock In {\em Singularities, algebraic geometry, commutative algebra, and related topics}, pages 191--202. Springer, Cham, 2018.

\bibitem[Dua14]{DuarSurf}
Daniel Duarte.
\newblock Nash modification on toric surfaces.
\newblock {\em Rev. R. Acad. Cienc. Exactas F\'{\i}s. Nat. Ser. A Mat. RACSAM}, 108(1):153--171, 2014.

\bibitem[Ful93]{fulton1993introduction}
William Fulton.
\newblock {\em Introduction to toric varieties}.
\newblock Number 131 in Annals of mathematics studies. Princeton university press, 1993.

\bibitem[GM12]{GrMi12}
Dima Grigoriev and Pierre~D. Milman.
\newblock Nash resolution for binomial varieties as {E}uclidean division. {A} priori termination bound, polynomial complexity in essential dimension 2.
\newblock {\em Adv. Math.}, 231(6):3389--3428, 2012.

\bibitem[GPT14]{GPTe14}
Pedro~D. Gonz\'alez~P\'erez and Bernard Teissier.
\newblock Toric geometry and the {S}emple-{N}ash modification.
\newblock {\em Rev. R. Acad. Cienc. Exactas F\'is. Nat. Ser. A Mat. RACSAM}, 108(1):1--48, 2014.

\bibitem[GS77]{GS1}
Gerardo Gonzalez~Sprinberg.
\newblock \'{E}ventails en dimension {$2$} et transform\'{e} de {N}ash.
\newblock {\em Publications du Centre de Math\'{e}matiques de l'E.N.S.}, pages 1--68, 1977.

\bibitem[GS82]{GS2}
Gerardo Gonzalez-Sprinberg.
\newblock R\'{e}solution de {N}ash des points doubles rationnels.
\newblock {\em Ann. Inst. Fourier (Grenoble)}, 32(2):x, 111--178, 1982.

\bibitem[GS09]{GS3}
Gerard Gonzalez-Sprinberg.
\newblock On {N}ash blow-up of orbifolds.
\newblock In {\em Singularities---{N}iigata--{T}oyama 2007}, volume~56 of {\em Adv. Stud. Pure Math.}, pages 133--149. Math. Soc. Japan, Tokyo, 2009.

\bibitem[Hir83]{Hi83}
Heisuke Hironaka.
\newblock On {N}ash blowing-up.
\newblock In {\em Arithmetic and geometry, {V}ol. {II}}, volume~36 of {\em Progr. Math.}, pages 103--111. Birkh\"auser, Boston, MA, 1983.

\bibitem[Nob75]{Nob75}
A.~Nobile.
\newblock Some properties of the {N}ash blowing-up.
\newblock {\em Pacific J. Math.}, 60(1):297--305, 1975.

\bibitem[Oda83]{oda1983convex}
Tadao Oda.
\newblock {\em Convex bodies and algebraic geometry: an introduction to the theory of toric varieties}.
\newblock Springer, 1983.

\bibitem[Reb77]{Reb}
Vaho Rebassoo.
\newblock Desingularization properties of the {N}ash blowing-up process, 1977.
\newblock Ph. D. dissertation, University of Washington.

\bibitem[{Sag}24]{sagemath}
{Sage Developers}.
\newblock {\em {S}ageMath, the {S}age {M}athematics {S}oftware {S}ystem ({V}ersion 10.4)}, 2024.
\newblock {\tt https://www.sagemath.org}.

\bibitem[Sem54]{Se54}
J.~G. Semple.
\newblock Some investigations in the geometry of curve and surface elements.
\newblock {\em Proc. London Math. Soc. (3)}, 4:24--49, 1954.

\bibitem[Spi90]{Sp90}
Mark Spivakovsky.
\newblock Sandwiched singularities and desingularization of surfaces by normalized {N}ash transformations.
\newblock {\em Ann. of Math. (2)}, 131(3):411--491, 1990.

\bibitem[Spi20]{Sp20}
Mark Spivakovsky.
\newblock Resolution of singularities: an introduction.
\newblock In {\em Handbook of geometry and topology of singularities. {I}}, pages 183--242. Springer, Cham, [2020] \copyright 2020.

\end{thebibliography}

\end{document}